\RequirePackage{etoolbox}
\csdef{input@path}{{style/}{graphics/}}
\makeatletter
\input{arxiv-vmsta.cfg}
\makeatother
\documentclass[numbers,compress]{vmsta}
\usepackage{dcolumn}

\volume{1}
\pubyear{2014}
\firstpage{167}
\lastpage{180}
\doi{10.15559/15-VMSTA18}

\startlocaldefs

\newtheorem{thm}{Theorem}
\newtheorem{lemma}{Lemma}

\newtheorem{prop}{Proposition}

\theoremstyle{definition}
\newtheorem{remark}{Remark}

\newtheorem{ex}{Example}

\newcolumntype{d}[1]{D{.}{.}{#1}}

\allowdisplaybreaks
\endlocaldefs

\begin{document}
\begin{frontmatter}

\title{Practical approaches to the estimation
of the ruin probability in a risk model
with additional funds}

\author[a]{\inits{Yu.}\fnm{Yuliya}\snm{Mishura}}\email{myus@univ.kiev.ua}
\author[a]{\inits{O.}\fnm{Olena}\snm{Ragulina}\corref{cor1}}\email
{ragulina.olena@gmail.com}
\cortext[cor1]{Corresponding author.}
\author[b]{\inits{O.}\fnm{Oleksandr}\snm{Stroyev}}\email{o.stroiev@chnu.edu.ua}

\address[a]{Taras Shevchenko National University of Kyiv,\\
Department of Probability Theory, Statistics and Actuarial Mathematics,\\
64 Volodymyrska Str., 01601 Kyiv, Ukraine}
\address[b]{Yuriy Fedkovych Chernivtsi National University,\\
Department of Mathematical Modelling,\\
2 Kotsjubynskyi Str., 58012 Chernivtsi, Ukraine}

\markboth{Yu. Mishura et al.}{Practical approaches to
the estimation of the ruin probability}

\begin{abstract}
We deal with a generalization of the classical risk model when an
insurance company
gets additional funds whenever a claim arrives and consider some
practical approaches
to the estimation of the ruin probability. In particular, we get an
upper exponential
bound and construct an analogue to the De Vylder approximation for the
ruin probability.
We compare results of these approaches with statistical estimates
obtained by
the Monte Carlo method for selected distributions of claim sizes and
additional funds.
\end{abstract}


\received{4 January 2015}
\accepted{19 January 2015}
\publishedonline{2 February 2015}
\end{frontmatter}

\section{Introduction}
\label{se:1}

Let $(\varOmega, \mathfrak{F}, \mathbb{P})$ be a probability space satisfying
the usual conditions, and let all the objects be defined on it.
We deal with the risk model that generalizes the classical one and was
considered
in~\cite{Mishura_Ragulina_Stroyev}.

In the classical risk model (see, e.g.,~\cite{Asmussen, Grandell,
Rolski_Schmidli_Schmidt_Teugels}), an insurance company has an initial
surplus $x\ge0$ and
receives premiums with constant intensity $c>0$. Claim sizes form a sequence
$(\xi_i)_{i\ge1}$ of nonnegative i.i.d.\ random variables with c.d.f.\
$F_1(y)=\mathbb P[\xi_i\le y]$ and finite expectation $\mathbb E[\xi
_i]=\mu_1$.
The number of claims on the time interval $[0,t]$ is a homogeneous
Poisson process
$(N_t)_{t\ge0}$ with intensity $\lambda>0$.

In addition to the classical risk model, we suppose that the insurance company
gets additional funds $\eta_i$ when the $i$th claim arrives.
These funds can be considered, for instance, as additional investment income,
which does not depend on the surplus of the company.
We assume that $(\eta_i)_{i\ge1}$ is a sequence of nonnegative i.i.d.
random variables with c.d.f. $F_2(y)=\mathbb P[\eta_i\le y]$ and finite
expectation
$\mathbb E[\eta_i]=\mu_2$. The sequences $(\xi_i)_{i\ge1}$, $(\eta
_i)_{i\ge1}$
and the process $(N_t)_{t\ge0}$ are mutually independent.
Let $(\mathfrak{F}_t)_{t\ge0}$ be the filtration generated by $(\xi
_i)_{i\ge1}$,
$(\eta_i)_{i\ge1}$, and $(N_t)_{t\ge0}$.

Let $X_t(x)$ be the surplus of the insurance company at time $t$,
provided that
its initial surplus is $x$. Then the surplus process $ (X_t(x)
)_{t\ge0}$
is defined as
\begin{equation}
\label{eq:1} X_t(x) =x+ct-\sum_{i=1}^{N_t}(
\xi_i-\eta_i), \quad t\ge0.
\end{equation}
Note that we set $\sum_{i=1}^{0}(\xi_i-\eta_i)=0$ in~\eqref{eq:1} if $N_t=0$.

The ruin time is defined as
\[
\tau(x)= \inf\bigl\{t\ge0\colon X_t(x) <0\bigr\}.
\]
We suppose that $\tau(x) =\infty$ if $X_t (x) \ge0$ for all $t\ge0$.
The infinite-horizon ruin probability is given by
\[
\psi(x)= \mathbb P \bigl[\inf\nolimits_{t\ge0}\; X_t(x) <0
\bigr],
\]
which is equivalent to
\[
\psi(x)= \mathbb P \bigl[\tau(x)< \infty\bigr].
\]
The corresponding infinite-horizon survival probability equals
\[
\varphi(x)= 1-\psi(x).
\]

Note that the ruin never occurs if $\mathbb P[\xi_i-\eta_i \le0]=1$.
If $\mathbb P[\xi_i-\eta_i \ge0]=1$,\allowbreak then we deal with the classical
risk model.
So in what follows, we assume that\allowbreak $\mathbb P[\xi_i-\eta_i>0]>0$ and
$\mathbb P[\xi_i-\eta_i<0]>0$.
In this case, if $c-\lambda\mu_1+\lambda\mu_2 \le0$, then $\varphi
(x)=0$ for all $x\ge0$;
if $c-\lambda\mu_1+\lambda\mu_2 >0$, then $\lim_{x\to+\infty} \varphi(x)=1$
(see \cite[Lemma~2.1]{Mishura_Ragulina_Stroyev}).

In this paper, we consider some practical approaches to the estimation
of the ruin
probability. In particular, we get an upper exponential bound and construct
an analogue to the De Vylder approximation for the ruin probability. Moreover,
we compare results of these approaches with statistical estimates
obtained by
the Monte Carlo method for selected distributions of claim sizes and
additional funds.

Paper~\cite{Mishura_Ragulina_Stroyev}, where this risk model is
considered, is devoted to the investigation of continuity and
differentiability of the infinite-horizon survival probability and
derivation of an integro-differential equation for this function.
When claim sizes and additional funds are exponentially distributed,
a closed-form solution to this equation can be found.

\begin{thm}[\cite{Mishura_Ragulina_Stroyev}, Theorem~4.1]
\label{thm:1}
Let the surplus process $ (X_t(x) )_{t\ge0}$ follow~\eqref{eq:1}
under the above assumptions, the random variables $\xi_i$ and $\eta_i$,
$i\ge1$,
be exponentially distributed with means $\mu_1$ and $\mu_2$ correspondingly,
and $c-\lambda\mu_1+\lambda\mu_2 >0$. Then
\begin{equation}
\label{eq:2} \varphi(x)=1+\frac{\lambda\mu_1(1-\alpha\mu_2)}{
(c\alpha-\lambda)(1-\alpha\mu_2)(\mu_1+\mu_2)+\lambda\mu_2}\,e^{\alpha x}
\end{equation}
for all $x\ge0$, where
\[
\alpha=\frac{\lambda\mu_1\mu_2+c\mu_1-c\mu_2
-\sqrt{c^2(\mu_1^2+\mu_2^2)+\lambda^2\mu_1^2\mu_2^2
+2c\mu_1\mu_2(c-\lambda\mu_1+\lambda\mu_2)}}{
2c\mu_1\mu_2}.
\]
\end{thm}

\begin{remark}[\cite{Mishura_Ragulina_Stroyev}, Remark~4.1]
\label{remark:1}
It is justified in the proof of Theorem~\ref{thm:1} that $\alpha<0$ and
\[
-1<\frac{\lambda\mu_1(1-\alpha\mu_2)}{
(c\alpha-\lambda)(1-\alpha\mu_2)(\mu_1+\mu_2)+\lambda\mu_2}<0.
\]
So the function $\varphi(x)$ defined by~\eqref{eq:2} satisfies all the natural
properties of the survival probability. In particular, this function is
nondecreasing and bounded by 0 from below and by 1 from above.
\end{remark}

It is well known that even for the classical risk model, there are only
a few cases
where an analytic expression for the survival probability can be found.
So numerous approximations have been considered and investigated
for the classical risk model (see, e.g.,~\cite{Asmussen, Beekman, Choi,
DeVylder, Grandell, Rolski_Schmidli_Schmidt_Teugels}).
``Simple approximations'' form a special class of approximations for
the ruin or survival probabilities. They use only some moments of the
distribution
of claim sizes and do not take into account the detailed tail behavior
of that
distribution. Such approximations may be based on limit theorems or on heuristic
arguments. The most successful ``simple approximation'' is certainly
the De Vylder
approximation~\cite{DeVylder}, which is based on the heuristic idea to replace
the risk process with a risk process with exponentially distributed
claim sizes
such that the first three moments coincide (see also~\cite{Grandell,
Rolski_Schmidli_Schmidt_Teugels}). This approximation is known to work extremely
well for some distributions of claim sizes. Later, Grandell analyzed
the De Vylder
approximation and other ``simple approximations'' from a more
mathematical point
of view and gave a possible explanation why the De Vylder approximation
is so good
(see~\cite{Grandell2000}).

We deal with the case where the claim sizes have a light-tailed distribution.
The rest of the paper is organized as follows.
In Section~\ref{se:2}, we get an upper exponential bound for the ruin
probability,
which is an analogue of the famous Lundberg inequality.
Section~\ref{se:3} is devoted to the construction of an analogue of the
De Vylder
approximation. In Section~\ref{se:4}, we give a simple formula that relates
the accuracy and reliability of the approximation of the ruin
probability by
its statistical estimate obtained by the Monte Carlo method. In
Section~\ref{se:5},
we compare the results of these approaches for some distributions of
claim sizes
and additional funds. Section~\ref{se:6} concludes the paper.

\section{Exponential bound}
\label{se:2}

To get an upper exponential bound for the ruin probability, we use the
martingale
approach introduced by Gerber~\cite{Gerber} (see also~\cite{Boikov, Grandell,
Rolski_Schmidli_Schmidt_Teugels}).

Let\begingroup\abovedisplayskip=5pt\belowdisplayskip=5pt
\[
U_t=ct-\sum_{i=1}^{N_t}(
\xi_i-\eta_i), \quad t\ge0.
\]

For all $R\ge0$, we define the exponential process $ (V_t(R)
)_{t\ge0}$
by
\[
V_t(R)=e^{-RU_t}.
\]

\begin{lemma}
\label{lemma:1}
If there is $\hat R>0$ such that
\begin{equation}
\label{eq:3} \lambda \Biggl(\int_0^{+\infty}
e^{\hat Ry}\, dF_1(y) \cdot\int_0^{+\infty}
e^{-\hat Ry}\, dF_2(y) -1 \Biggr)=c\hat R,
\end{equation}
then $ (V_t(\hat R) )_{t\ge0}$ is an $(\mathfrak{F}_t)$-martingale.
\end{lemma}

\begin{proof}
For all $R>0$ such that $\mathbb E  [e^{R\xi_i} ]<\infty$, if any,
we have
\begin{equation}\label{eq:4} %
\begin{split} \mathbb E \bigl[V_t(R)\bigr]
&=e^{-cRt}\, \mathbb E \Biggl[\exp \Biggl\{R\sum
_{i=1}^{N_t}(\xi_i-\eta_i)
\Biggr\} \Biggr]
\\
&=e^{-cRt}\,\sum_{j=0}^{\infty}e^{-\lambda t}
\frac{(\lambda t)^j}{j!} \bigl(\mathbb E \bigl[e^{R(\xi_i-\eta_i)} \bigr]
\bigr)^j
\\
&=\exp \bigl\{t \bigl(\lambda\mathbb E \bigl[e^{R(\xi_i-\eta_i)} \bigr] -\lambda-cR
\bigr) \bigr\}. \end{split} %
\end{equation}

If there is $\hat R>0$ such that~\eqref{eq:3} holds, then
$\mathbb E  [e^{\hat R\xi_i} ]<\infty$, and
for all $t_2\ge t_1\ge0$, we have
\begin{equation*}
\begin{split} \mathbb E \bigl[V_{t_2}(\hat R) \,/\,
\mathfrak{F}_{t_1}\bigr] &=\mathbb E \Biggl[\exp \Biggl\{-\hat R
\Biggl(ct_2 -\sum_{i=1}^{N_{t_2}}(
\xi_i-\eta_i) \Biggr) \Biggr\} \, \vphantom{ ( )} \Big/\,
\mathfrak{F}_{t_1} \Biggr]
\\
&=\mathbb E \Biggl[\exp \Biggl\{-\hat R \Biggl(ct_1 -\sum
_{i=1}^{N_{t_1}}(\xi_i-\eta_i)
\Biggr) \Biggr\} \Biggr]
\\
&\quad\times\mathbb E \Biggl[\exp \Biggl\{-\hat R \Biggl(c(t_2-t_1)
-\sum_{i=N_{t_1}}^{N_{t_2}}(\xi_i-
\eta_i) \Biggr) \Biggr\} \, \vphantom{ ( )} \Big/\, \mathfrak{F}_{t_1}
\Biggr]
\\
&=\mathbb E \bigl[V_{t_1}(\hat R)\bigr] \cdot\mathbb E \Biggl[\exp
\Biggl\{-\hat R \Biggl(c(t_2-t_1) -\sum
_{i=1}^{N_{t_2-t_1}}(\xi_i-\eta_i)
\Biggr) \Biggr\} \Biggr]
\\
&=\mathbb E \bigl[V_{t_1}(\hat R)\bigr]. \end{split} %
\end{equation*}
Here we used the fact that
\begin{equation*}
\begin{split} &\mathbb E \Biggl[\exp \Biggl\{-\hat R
\Biggl(c(t_2-t_1) -\sum_{i=1}^{N_{t_2-t_1}}(
\xi_i-\eta_i) \Biggr) \Biggr\} \Biggr]
\\
&\quad
=\exp \bigl\{ (t_2-t_1 ) \bigl(\lambda \mathbb E
\bigl[e^{\hat R(\xi_i-\eta_i)} \bigr]-\lambda-c\hat R \bigr) \bigr\}=1 \end{split}
\end{equation*}
by~\eqref{eq:3} and~\eqref{eq:4}.

Thus, $ (V_t(\hat R) )_{t\ge0}$ is an $(\mathfrak{F}_t)$-martingale,
which is the desired conclusion.
\end{proof}\endgroup

\begin{thm}
\label{thm:2}
If there is $\hat R>0$ such that~\eqref{eq:3} holds, then for all $x\ge
0$, we have
\begin{equation}
\label{eq:5} \psi(x)\le e^{-\hat Rx}.
\end{equation}
\end{thm}

\begin{proof}
It is easily seen that $\tau(x)$ is an $(\mathfrak{F}_t)$-stopping time.
Hence, $\tau(x)\wedge T$ is a bounded $(\mathfrak{F}_t)$-stopping time
for any fixed $T\ge0$. The process $ (V_t(\hat R) )_{t\ge0}$
is an $(\mathfrak{F}_t)$-martingale by Lemma~\ref{lemma:1}. Moreover,
$ (V_t(\hat R) )_{t\ge0}$ is positive a.s.\ by its definition.
Consequently, applying the optional stopping theorem yields
\begin{equation*}
\begin{split} 1&=V_0(\hat R)=\mathbb E
\bigl[V_{\tau(x)\wedge T}(\hat R)\bigr]
\\
&=\mathbb E \bigl[V_{\tau(x)}(\hat R) \cdot\mathbb I_{\{\tau(x)<T\}}\bigr] +
\mathbb E \bigl[V_T(\hat R) \cdot\mathbb I_{\{\tau(x)\ge T\}}\bigr]
\\
&\ge\mathbb E \bigl[V_{\tau(x)}(\hat R) \cdot\mathbb I_{\{\tau(x)<T\}}\bigr]
\\
&=\mathbb E \Biggl[\exp \Biggl\{-\hat R \Biggl(c\tau(x) -\sum
_{i=1}^{N_{\tau(x)}}(\xi_i-\eta_i)
\Biggr) \Biggr\} \cdot\mathbb I_{\{\tau(x)<T\}} \Biggr]
\\
&\ge e^{\hat Rx} \cdot\mathbb P \bigl[\tau(x)<T\bigr], \end{split}
\end{equation*}
where $\mathbb I_{\{\cdot\}}$ is the indicator of an event. This gives
\begin{equation}
\label{eq:6} \mathbb P \bigl[\tau(x)<T\bigr]\le e^{-\hat Rx}
\end{equation}
for all $T\ge0$. Letting $T\to\infty$ in~\eqref{eq:6} yields
\[
\mathbb P \bigl[\tau(x)<\infty\bigr]\le e^{-\hat Rx},
\]
which is our assertion.
\end{proof}

\begin{ex}
\label{ex:1}
Let the random variables $\xi_i$ and $\eta_i$, $i\ge1$,
be exponentially distributed with means $\mu_1$ and $\mu_2$, respectively.
Then~\eqref{eq:3} can be rewritten as
\[
\lambda \biggl(\frac{1}{(1-\mu_1 \hat R)(1+\mu_2 \hat R)} -1 \biggr)=c\hat R,
\]
where $\hat R\in(0,1/\mu_1)$. This condition is equivalent to
\begin{equation}
\label{eq:7} c\mu_1\mu_2\hat R^3+(\lambda
\mu_1\mu_2+c\mu_1-c\mu_2)\hat
R^2 -(c-\lambda\mu_1+\lambda\mu_2)\hat R=0.
\end{equation}

If $c-\lambda\mu_1+\lambda\mu_2 >0$, then
\[
c^2\bigl(\mu_1^2+\mu_2^2
\bigr)+\lambda^2\mu_1^2\mu_2^2+2c
\mu_1\mu_2(c-\lambda\mu _1+\lambda
\mu_2)>0.
\]
So there are three real solutions to~\eqref{eq:7}. They are
\[
\hat R_1=0,
\]
\[
\hat R_2=-\frac{\lambda\mu_1\mu_2+c\mu_1-c\mu_2
-\sqrt{A(c,\lambda,\mu_1,\mu_2)}}{
2c\mu_1\mu_2},
\]
\[
\hat R_3=-\frac{\lambda\mu_1\mu_2+c\mu_1-c\mu_2
+\sqrt{A(c,\lambda,\mu_1,\mu_2)}}{
2c\mu_1\mu_2},
\]
where
\[
A(c,\lambda,\mu_1,\mu_2) =c^2\bigl(
\mu_1^2+\mu_2^2\bigr)+
\lambda^2\mu_1^2\mu_2^2+2c
\mu_1\mu_2(c-\lambda\mu _1+\lambda
\mu_2).
\]

Furthermore, it is easy to check that, in this case,
\[
|\lambda\mu_1\mu_2+c\mu_1-c
\mu_2| <\sqrt{A(c,\lambda,\mu_1,\mu_2)}.
\]
From this we conclude that $\hat R_2>0$ and $\hat R_3<0$.
Since
\[
A(c,\lambda,\mu_1,\mu_2) <(\lambda\mu_1
\mu_2+c\mu_1+c\mu_2)^2,
\]
we have
\[
\hat R_2<\frac{(\lambda\mu_1\mu_2+c\mu_1+c\mu_2)-(\lambda\mu_1\mu_2+c\mu
_1-c\mu_2)}{
2c\mu_1\mu_2}<\frac{1}{\mu_1}.
\]
Hence, $\hat R_2$ is a unique positive solution to~\eqref{eq:7}, and
an exponential bound~\eqref{eq:5} can be rewritten as follows:
\begin{equation}
\label{eq:8} \psi(x)\le e^{-\hat R_2 x}.
\end{equation}
Comparing~\eqref{eq:8} with~\eqref{eq:2}, we see that the exponential bound
and the analytic expression for the ruin probability differ
in a constant multiplier only.

If $c-\lambda\mu_1+\lambda\mu_2 \le0$, then $\mu_1>\mu_2$,
which gives $\lambda\mu_1\mu_2+c\mu_1-c\mu_2>0$.
Let $\hat R_2$ and $\hat R_3$ be two nonzero solutions to~\eqref{eq:7}.
Applying Vieta's formulas yields $\hat R_2 +\hat R_3<0$ and
$\hat R_2 \hat R_3>0$. Consequently, if $\hat R_2$ and $\hat R_3$
are real, they are negative. Thus, \eqref{eq:7} has no positive
solution, and Theorem~\ref{thm:2} does not give us an exponential bound
for the ruin probability.
Indeed, in this case, $\psi(x)=1$ for all $x\ge0$
(see \cite[Lemma~2.1]{Mishura_Ragulina_Stroyev}).
\end{ex}

\section{Analogue to the De Vylder approximation}
\label{se:3}

To construct an analogue to the De Vylder approximation, we replace the process
$ (U_t )_{t\ge0}$ with a process $ (\tilde U_t )_{t\ge
0}$ with
exponentially distributed claim sizes such that
\begin{equation}
\label{eq:9} \mathbb E \bigl[U_t^k\bigr]=\mathbb E
\bigl[\tilde U_t^k\bigr], \quad k=1,2,3.
\end{equation}

Since the process $ (\tilde U_t )_{t\ge0}$ in this risk model
is determined
by the four parameters $(\tilde c,\tilde\lambda,\tilde\mu_1,\tilde\mu_2)$
in contrast to the classical risk model, where it is determined by
three parameters,
we use the additional condition
\begin{equation}
\label{eq:10} \frac{\mu_1}{\mu_2}=\frac{\tilde\mu_1}{\tilde\mu_2}.
\end{equation}
Note that we could have used the condition $\mathbb E [U_t^4]=\mathbb E
[\tilde U_t^4]$
instead of~\eqref{eq:10}, but it would have led to tedious calculations and
solving polynomial equations of higher degree.

Let $(\tilde\xi_i)_{i\ge1}$ be a sequence of i.i.d. random variables
exponentially distributed with mean $\tilde\mu_1$. Similarly, let
$(\tilde\eta_i)_{i\ge1}$ be a sequence of i.i.d. random variables
exponentially distributed with mean $\tilde\mu_2$.
An easy computation shows that
\begin{equation}
\label{eq:11} \mathbb E \bigl[\tilde\xi_i^k\bigr]=k!
\tilde\mu_1^k \quad\text{and} \quad \mathbb E \bigl[
\tilde\eta_i^k\bigr]=k!\tilde\mu_2^k.
\end{equation}

Let $\mathbb E  [\xi_i^3 ]<\infty$ and $\mathbb E  [\eta
_i^3 ]<\infty$.
Then we have
\[
\mathbb E [U_t]=ct-\mathbb E \Biggl[\sum
_{i=1}^{N_t}(\xi_i-\eta_i)
\Biggr] =ct-\lambda t\, \mathbb E [\xi_i-\eta_i ],
\]%
\begin{equation*}
\begin{split} \mathbb E \bigl[U_t^2
\bigr]&=(ct)^2-2ct\, \mathbb E \Biggl[\sum
_{i=1}^{N_t}(\xi _i-\eta_i)
\Biggr] +\mathbb E \Biggl[ \Biggl(\sum_{i=1}^{N_t}(
\xi_i-\eta_i) \Biggr)^2 \Biggr]
\\
&=(ct)^2-2ct\cdot\lambda t\, \mathbb E [\xi_i-
\eta_i ] +\lambda t\, \mathbb E \bigl[(\xi_i-
\eta_i)^2 \bigr] +(\lambda t)^2 \bigl(\mathbb
E [\xi_i-\eta_i ] \bigr)^2
\\
&=\lambda t\, \mathbb E \bigl[(\xi_i-\eta_i)^2
\bigr] + \bigl(ct-\lambda t\, \mathbb E [\xi_i-\eta_i ]
\bigr)^2
\\
&=\lambda t\, \mathbb E \bigl[(\xi_i-\eta_i)^2
\bigr] + \bigl(\mathbb E [U_t] \bigr)^2, \end{split}
\end{equation*}
\begin{equation*}
\begin{split} \mathbb E \bigl[U_t^3
\bigr]&=(ct)^3-3(ct)^2 \, \mathbb E \Biggl[\sum
_{i=1}^{N_t}(\xi_i-\eta_i)
\Biggr]
\\
&\quad+3ct\, \mathbb E \Biggl[ \Biggl(\sum_{i=1}^{N_t}(
\xi_i-\eta _i) \Biggr)^2 \Biggr] -\mathbb E
\Biggl[ \Biggl(\sum_{i=1}^{N_t}(
\xi_i-\eta_i) \Biggr)^3 \Biggr]
\\
&=(ct)^3-3(ct)^2\cdot\lambda t\, \mathbb E [
\xi_i-\eta_i ] +3ct (\lambda t\, \mathbb E \bigl[(
\xi_i-\eta_i)^2 \bigr]
\\
&\quad+(\lambda t)^2 \bigl(\mathbb E [\xi_i-
\eta_i ] \bigr)^2 -\lambda t\, \mathbb E \bigl[(
\xi_i-\eta_i)^3 \bigr] -3(\lambda
t)^2\, \mathbb E \bigl[(\xi_i-\eta_i)^2
\bigr]
\\
&\quad+\lambda t\, \mathbb E \bigl[(\xi_i-\eta_i)^2
\bigr]\, \mathbb E \bigl[(\xi_i-\eta_i)^2
\bigr] -(\lambda t)^3 \bigl(\mathbb E [\xi_i-
\eta_i ] \bigr)^3
\\
&=-\lambda t\, \mathbb E \bigl[(\xi_i-\eta_i)^3
\bigr] + \bigl(ct-\lambda t\, \mathbb E [\xi_i-\eta_i ]
\bigr)^3
\\
&\quad+3\lambda t\, \mathbb E \bigl[(\xi_i-\eta_i)^2
\bigr] \bigl(ct-\lambda t\, \mathbb E [\xi_i-\eta_i ]
\bigr)
\\
&=-\lambda t\, \mathbb E \bigl[(\xi_i-\eta_i)^3
\bigr] + \bigl(\mathbb E [U_t] \bigr)^3 +3 \bigl(\mathbb
E \bigl[U_t^2\bigr]- \bigl(\mathbb E [U_t]
\bigr)^2 \bigr)\,\mathbb E [U_t]. \end{split} %
\end{equation*}

Applying similar arguments to the process $ (\tilde U_t )_{t\ge0}$,
we conclude that~\eqref{eq:9} is equivalent to
\begin{equation}
\label{eq:12}
\left\{ %
\begin{aligned}
ct-\lambda t\, \mathbb E [ \xi_i-\eta_i ] &=\tilde ct-\tilde\lambda t\, \mathbb E [
\tilde\xi_i-\tilde\eta _i ],
\\
\lambda t\, \mathbb E \bigl[(\xi_i-\eta_i)^2
\bigr] &=\tilde\lambda t\, \mathbb E \bigl[(\tilde\xi_i-\tilde
\eta_i)^2 \bigr],
\\
-\lambda t\, \mathbb E \bigl[(\xi_i-\eta_i)^3
\bigr] &=-\tilde\lambda t\, \mathbb E \bigl[(\tilde\xi_i-\tilde
\eta_i)^3 \bigr], \end{aligned}\right. %
\end{equation}

By~\eqref{eq:11} we can rewrite~\eqref{eq:12} as
\begin{equation}
\label{eq:13} \left\{ %
\begin{aligned} c-\lambda(\mu_1-
\mu_2)&=\tilde c-\tilde\lambda(\tilde\mu_1-\tilde\mu
_2),
\\
\lambda\,\mathbb E \bigl[(\xi_i-\eta_i)^2
\bigr] &=2\tilde\lambda\bigl(\tilde\mu_1^2-\tilde
\mu_1\tilde\mu_2+\tilde\mu_2^2
\bigr),
\\
\lambda\,\mathbb E \bigl[(\xi_i-\eta_i)^3
\bigr] &=6\tilde\lambda\bigl(\tilde\mu_1^3-\tilde
\mu_1^2\tilde\mu_2 +\tilde\mu_1
\tilde\mu_2^2-\tilde\mu_2^3\bigr).
\end{aligned} \right.%
\end{equation}

Substituting $\tilde\mu_2=\mu_2\tilde\mu_1/\mu_1$ into the second and
third equations
of system~\eqref{eq:13} yields
\begin{equation}
\label{eq:14} \lambda\,\mathbb E \bigl[(\xi_i-
\eta_i)^2 \bigr] =2\tilde\lambda\tilde
\mu_1^2 \biggl(1-\frac{\mu_2}{\mu_1}+\frac{\mu
_2^2}{\mu_1^2}
\biggr),
\end{equation}
\begin{equation}
\label{eq:15} \lambda\,\mathbb E \bigl[(\xi_i-
\eta_i)^3 \bigr] =6\tilde\lambda\tilde
\mu_1^3 \biggl(1-\frac{\mu_2}{\mu_1}+\frac{\mu
_2^2}{\mu_1^2} -
\frac{\mu_2^3}{\mu_1^3} \biggr).
\end{equation}

Dividing~\eqref{eq:15} by~\eqref{eq:14} gives
\begin{equation}
\label{eq:16} \tilde\mu_1=\frac{\mu_1 (\mu_1^2-\mu_1\mu_2+\mu_2^2 )\,
\mathbb E  [(\xi_i-\eta_i)^3 ]}{
3 (\mu_1^3-\mu_1^2\mu_2+\mu_1\mu_2^2-\mu_2^3 )\,
\mathbb E  [(\xi_i-\eta_i)^2 ]}.
\end{equation}

Consequently, we have
\begin{equation}
\label{eq:17} \tilde\mu_2=\frac{\mu_2 (\mu_1^2-\mu_1\mu_2+\mu_2^2 )\,
\mathbb E  [(\xi_i-\eta_i)^3 ]}{
3 (\mu_1^3-\mu_1^2\mu_2+\mu_1\mu_2^2-\mu_2^3 )\,
\mathbb E  [(\xi_i-\eta_i)^2 ]}.
\end{equation}

Substituting~\eqref{eq:16} into~\eqref{eq:14}, we get
\begin{equation}
\label{eq:18} \tilde\lambda=\frac{9\lambda
 (\mu_1^3-\mu_1^2\mu_2+\mu_1\mu_2^2-\mu_2^3 )^2\,
 (\mathbb E  [(\xi_i-\eta_i)^2 ] )^3}{
2 (\mu_1^2-\mu_1\mu_2+\mu_2^2 )^3\,
 (\mathbb E  [(\xi_i-\eta_i)^3 ] )^2}.
\end{equation}

Substituting~\eqref{eq:16}--\eqref{eq:18} into the first equation
of system~\eqref{eq:13}, we obtain
\begin{equation}
\label{eq:19} \tilde c =c-\lambda (\mu_1-\mu_2 )
\biggl(1-\frac{3 (\mu_1^3-\mu_1^2\mu_2+\mu_1\mu_2^2-\mu_2^3 )\,
 (\mathbb E  [(\xi_i-\eta_i)^2 ] )^2}{
2 (\mu_1^2-\mu_1\mu_2+\mu_2^2 )^2\,
\mathbb E  [(\xi_i-\eta_i)^3 ]} \biggr).
\end{equation}

Note that since $F_1(y)$ and $F_2(y)$ are known, it is easy to find
$\mathbb E  [(\xi_i-\eta_i)^2 ]$ and $\mathbb E  [(\xi
_i-\eta_i)^3 ]$
if $\mathbb E  [\xi_i^3 ]<\infty$ and $\mathbb E  [\eta
_i^3 ]<\infty$.

By~\eqref{eq:16} and~\eqref{eq:17}, $\tilde\mu_1$ and $\tilde\mu_2$ are
positive, provided that
\begin{equation}
\label{eq:20} \bigl(\mu_1^3-\mu_1^2
\mu_2+\mu_1\mu_2^2-
\mu_2^3 \bigr)\, \mathbb E \bigl[(\xi_i-
\eta_i)^3 \bigr]>0.
\end{equation}
If~\eqref{eq:20} holds, then $\mu_1 \ne\mu_2$. So $\tilde\lambda$ is
also positive.
Moreover, $\tilde c$ is positive, provided that
\begin{equation}
\label{eq:21} c-\lambda (\mu_1-\mu_2 ) \biggl(1-
\frac{3 (\mu_1^3-\mu_1^2\mu_2+\mu_1\mu_2^2-\mu_2^3 )\,
 (\mathbb E  [(\xi_i-\eta_i)^2 ] )^2}{
2 (\mu_1^2-\mu_1\mu_2+\mu_2^2 )^2\,
\mathbb E  [(\xi_i-\eta_i)^3 ]} \biggr)>0.
\end{equation}

Thus, we get the following result.

\begin{prop}[\xch{A}{a}n analogue to the De Vylder approximation]
\label{prop:1}
Let the surplus process $ (X_t(x) )_{t\ge0}$ follow~\eqref{eq:1}
under the above assumptions, $c-\lambda\mu_1+\lambda\mu_2 >0$,\allowbreak
$\mathbb E  [\xi_i^3 ]<\infty$,
$\mathbb E  [\eta_i^3 ]<\infty$,
and let conditions~\eqref{eq:20} and~\eqref{eq:21} hold.
Then the ruin probability is approximately equal to
\[
\psi_{DV}(x)=\frac{\tilde\lambda\tilde\mu_1(\tilde\alpha\tilde\mu_2-1)}{
(\tilde c \tilde\alpha-\tilde\lambda)(1-\tilde\alpha\tilde\mu_2)
(\tilde\mu_1+\tilde\mu_2)+\tilde\lambda\tilde\mu_2}\,e^{\tilde\alpha x}
\]
for all $x\ge0$, where\vadjust{\eject}
\[
\tilde\alpha=\frac{\tilde\lambda\tilde\mu_1\tilde\mu_2+\tilde c \tilde
\mu_1-\tilde c \tilde\mu_2
-\sqrt{\tilde c^2(\tilde\mu_1^2+\tilde\mu_2^2)+\tilde\lambda^2\tilde\mu
_1^2\tilde\mu_2^2
+2\tilde c \tilde\mu_1\tilde\mu_2(\tilde c-\tilde\lambda\tilde\mu
_1+\tilde\lambda\tilde\mu_2)}}{
2\tilde c \tilde\mu_1\tilde\mu_2}
\]
and the parameters $(\tilde c,\tilde\lambda,\tilde\mu_1,\tilde\mu_2)$
are defined by~\eqref{eq:16}--\eqref{eq:19}.
\end{prop}

\begin{remark}
\label{remark:2}
Note that $\tilde\alpha<0$ and
\[
\frac{\tilde\lambda\tilde\mu_1(\tilde\alpha\tilde\mu_2-1)}{
(\tilde c \tilde\alpha-\tilde\lambda)(1-\tilde\alpha\tilde\mu_2)
(\tilde\mu_1+\tilde\mu_2)+\tilde\lambda\tilde\mu_2}>0
\]
in Proposition~\ref{prop:1}.
Indeed, the parameters $(\tilde c,\tilde\lambda,\tilde\mu_1,\tilde\mu_2)$
are positive. Moreover, since $c-\lambda\mu_1+\lambda\mu_2 >0$, we have
$\tilde c-\tilde\lambda\tilde\mu_1 +\tilde\lambda\tilde\mu_2 >0$
by the first equation of system~\eqref{eq:13}.
Hence, Theorem~\ref{thm:1} and Remark~\ref{remark:1} give us the
desired conclusion.
\end{remark}

\begin{remark}
\label{remark:3}
If claim sizes and additional funds are exponentially distributed, then
it is easily seen from~\eqref{eq:16}--\eqref{eq:19} that $\psi(x)=\psi_{DV}(x)$.
\end{remark}

\section{Statistical estimate obtained by the Monte Carlo method}
\label{se:4}

Let $N$ be the total number of simulations of the surplus process $X_t(x)$,
and let $\hat\psi(x)$ be the corresponding statistical estimate obtained
by the Monte Carlo method. To get it, we divide the number of simulations
that lead to the ruin by the total number of simulations.

\begin{prop}
\label{prop:2}
Let the surplus process $ (X_t(x) )_{t\ge0}$ follow~\eqref{eq:1}
under the above assumptions. Then for any $\varepsilon>0$, we have
\begin{equation}
\label{eq:22} \mathbb P \bigl[ \bigl|\psi(x)-\hat\psi(x) \bigr|>\varepsilon \bigr]
\le2e^{-2\varepsilon^2 N}.
\end{equation}
\end{prop}

The assertion of Proposition~\ref{prop:2} follows immediately from
Hoeffding's inequality (see~\cite{Hoeffding}).

\begin{remark}
\label{remark:4}
Formula~\eqref{eq:22} relates the accuracy and reliability of the approximation
of the ruin probability by its statistical estimate obtained by the
Monte Carlo method.
It enables us to find the number of simulations $N$, which is necessary
in order
to calculate the ruin probability with the required accuracy and reliability.
An obvious shortcoming of the Monte Carlo method is a too large number
of simulations $N$.
In all examples in Section~\ref{se:5}, we assume that $\varepsilon
=0{.}001$ and
$2e^{-2\varepsilon^2 N}=0{.}001$. Consequently, $N=3\,800\,452$.
\end{remark}

\section{Comparison of results}
\label{se:5}

\subsection{Erlang distributions for claim sizes and additional funds}

Let the probability density functions of $\xi_i$ and $\eta_i$ be
\[
f_1(y)=\frac{k_1^{k_1}y^{k_1-1}e^{-k_1y/\mu_1}}{\mu_1^{k_1}(k_1-1)!} \quad\text{and} \quad f_2(y)=
\frac{k_2^{k_2}y^{k_2-1}e^{-k_2y/\mu_2}}{\mu_2^{k_2}(k_2-1)!}
\]
for $y\ge0$, respectively, where $k_1$ and $k_2$ are positive integers.

In what follows, $h_1(R)$ and $h_2(R)$, where $R\ge0$, denote the moment
generating functions of $\xi_i$ and $\eta_i$, respectively, that is,
\[
h_1(R)=\mathbb E \bigl[e^{R\xi_i} \bigr] \quad\text{and}
\quad h_2(R)=\mathbb E \bigl[e^{R\eta_i} \bigr].
\]

An easy computation shows that
\[
h_1(R)=\int_0^{+\infty}e^{Ry}
\,dF_1(y) = \biggl(\frac{k_1}{k_1-\mu_1 R} \biggr)^{k_1}, \quad0
\le R< \frac{k_1}{\mu_1},
\]
\[
h_2(R)=\int_0^{+\infty}e^{Ry}
\,dF_2(y) = \biggl(\frac{k_2}{k_2-\mu_2 R} \biggr)^{k_2}, \quad0
\le R< \frac{k_2}{\mu_2}.
\]

Moreover, for all $R\ge0$, we have
\[
\int_0^{+\infty}e^{-Ry}\,dF_2(y)
= \biggl(\frac{k_2}{k_2+\mu_2 R} \biggr)^{k_2}.
\]

Thus, condition~\eqref{eq:3} can be rewritten as
\begin{equation}
\label{eq:23} \lambda\, \biggl(\frac{k_1}{k_1-\mu_1 \hat R} \biggr)^{k_1}\,
\biggl(\frac{k_2}{k_2+\mu_2 \hat R} \biggr)^{k_2} =\lambda+c\hat R,
\end{equation}
where $0<\hat R<k_1/\mu_1$. Furthermore, we have
\[
\mathbb E [\xi_i ]=h'_1(0)=
\mu_1, \qquad\mathbb E [\eta_i ]=h'_2(0)=
\mu_2,
\]
\[
\mathbb E \bigl[\xi_i^2 \bigr]=h''_1(0)=
\frac{(k_1+1)\mu_1^2}{k_1}, \qquad\mathbb E \bigl[\eta_i^2
\bigr]=h''_2(0)=\frac{(k_2+1)\mu_2^2}{k_2},
\]
\[
\mathbb E \bigl[\xi_i^3 \bigr]=h'''_1(0)=
\frac{(k_1+1)(k_1+2)\mu_1^3}{k_1^2}, \qquad\mathbb E \bigl[\eta_i^3
\bigr]=h'''_2(0)=
\frac{(k_2+1)(k_2+2)\mu
_2^3}{k_2^2}.
\]

Hence, we get
\begin{equation*}
\begin{split} \mathbb E \bigl[(\xi_i-
\eta_i)^2 \bigr] &=\mathbb E \bigl[\xi_i^2
\bigr]-2\mathbb E [\xi_i ]\mathbb E [\eta_i ] +\mathbb
E \bigl[\eta_i^2 \bigr]
\\
&=\frac{(k_1+1)\mu_1^2}{k_1}-2\mu_1\mu_2+\frac{(k_2+1)\mu_2^2}{k_2},
\end{split} %
\end{equation*}
\begin{equation*}
\begin{split} \mathbb E \bigl[(\xi_i-
\eta_i)^3 \bigr] &=\mathbb E \bigl[\xi_i^3
\bigr]-3\mathbb E \bigl[\xi_i^2 \bigr]\mathbb E [
\eta_i ] +3\mathbb E [\xi_i ]\mathbb E \bigl[
\eta_i^2 \bigr]-\mathbb E \bigl[\eta_i^3
\bigr]
\\
&=\frac{(k_1+1)(k_1+2)\mu_1^3}{k_1^2}-\frac{3(k_1+1)\mu_1^2\mu_2}{k_1}
\\
&\quad+\frac{3(k_2+1)\mu_2^2\mu_1}{k_2}-\frac{(k_2+1)(k_2+2)\mu_2^3}{k_2^2}. \end{split} %
\end{equation*}
Substituting $\mathbb E  [(\xi_i-\eta_i)^2 ]$ and $\mathbb E
 [(\xi_i-\eta_i)^3 ]$
into~\eqref{eq:16}--\eqref{eq:19}, we obtain the parameters
$(\tilde c,\tilde\lambda,\tilde\mu_1,\tilde\mu_2)$.

\begin{ex}
\label{ex:2}
Let $c=10$, $\lambda=4$, $\mu_1=2$, $\mu_2=0{.}5$, $k_1=3$, $k_2=2$. Then
$\hat R \approx0{.}349093$, which may not be an unique positive
solution to~\eqref{eq:23},
and
\[
\psi_{DV}(x)=0{.}612268\,e^{-0{.}332472\,x}.
\]

The results of computations are given in Table~\ref{table:1}.

\begin{table}
\caption{Results of computations: Erlang distributions for claim sizes
and additional funds}\label{table:1}
\begin{tabular*}{\textwidth}{@{\extracolsep{\fill
}}D{.}{.}{2.0}D{.}{.}{1.6}D{.}{.}{1.6}D{.}{.}{3.13}D{.}{.}{1.6}D{.}{.}{2.13}@{}}
\hline
\multicolumn{1}{@{}l}{$x$} &
\multicolumn{1}{l}{$\hat\psi(x)$} &
\multicolumn{1}{l}{$\psi_{DV}(x)$} &
\multicolumn{1}{l}{$ \bigl(\frac{\psi_{DV}(x)}{\hat\psi(x)}-1\bigr)\cdot100\%$} &
\multicolumn{1}{l}{$e^{-\hat Rx}$} &
\multicolumn{1}{l@{}}{$ \bigl(\frac{e^{-\hat Rx}}{\hat\psi(x)}-1\bigr)\cdot100\%$} \\
\hline
\rule{0pt}{9pt}0 &0.634149 &0.612268 &-3.45\% &1.000000 &57.69\%\\
1 &0.492768 &0.439087 &-10.89\% &0.705327 &43.14\%\\
2 &0.355769 &0.314891 &-11.49\% &0.497487 &39.83\%\\
5 &0.137737 &0.116142 &-15.68\% &0.174564 &26.74\%\\
10 & 0.023224 &0.022031 &-5.14\% &0.030473 &31.21\%\\
\hline
\end{tabular*}
\end{table}
\end{ex}

\subsection{Hyperexponential distributions for claim sizes and
additional funds}

Let
\[
F_1(y)=p_{1,1}F_{1,1}(y)+p_{1,2}F_{1,2}(y)+
\cdots+p_{1,k_1}F_{1,k_1}(y), \quad y\ge0,
\]
where $k_1\ge1$, $p_{1,j}>0$, $\sum_{j=1}^{k_1}p_{1,j}=1$,
$\sum_{j=1}^{k_1}p_{1,j}\,\mu_{1,j}=\mu_1$, and $F_{1,j}$ is the c.d.f.
of the exponential distribution with mean $\mu_{1,j}$;
\[
F_2(y)=p_{2,1}F_{2,1}(y)+p_{2,2}F_{2,2}(y)+
\cdots+p_{2,k_2}F_{2,k_2}(y), \quad y\ge0,
\]
where $k_2\ge1$, $p_{2,j}>0$, $\sum_{j=1}^{k_2}p_{2,j}=1$,
$\sum_{j=1}^{k_2}p_{2,j}\,\mu_{2,j}=\mu_2$, and $F_{2,j}$ is the c.d.f.
of the exponential distribution with mean $\mu_{2,j}$.

It is easy to check that
\[
h_1(R)=\sum_{j=1}^{k_1}
\frac{p_{1,j}}{1-\mu_{1,j}R}, \quad0\le R< \min \biggl\{\frac{1}{\mu_{1,1}},
\frac{1}{\mu_{1,2}},\dots ,\frac{1}{\mu_{1,k_1}} \biggr\},
\]
\[
h_2(R)=\sum_{j=1}^{k_2}
\frac{p_{2,j}}{1-\mu_{2,j}R}, \quad0\le R< \min \biggl\{\frac{1}{\mu_{2,1}},
\frac{1}{\mu_{2,2}},\dots ,\frac{1}{\mu_{2,k_2}} \biggr\}.
\]

Furthermore, for all $R\ge0$, we have
\[
\int_0^{+\infty}e^{-Ry}\,dF_2(y)
=\sum_{j=1}^{k_2}\frac{p_{2,j}}{1+\mu_{2,j}R}.
\]

Hence, condition~\eqref{eq:3} can be rewritten as
\begin{equation}
\label{eq:24} \lambda\, \Biggl(\sum_{j=1}^{k_1}
\frac{p_{1,j}}{1-\mu_{1,j}\hat R} \cdot\sum_{j=1}^{k_2}
\frac{p_{2,j}}{1+\mu_{2,j}\hat R} \Biggr) =\lambda+c\hat R,
\end{equation}
where $0<\hat R< \min \{1/\mu_{1,1},1/\mu_{1,2},\dots,1/\mu
_{1,k_1} \}$.
Moreover, we have
\[
\mathbb E [\xi_i ]=h'_1(0)=\sum
_{j=1}^{k_1}p_{1,j}\,\mu _{1,j}=
\mu_1, \qquad\mathbb E [\eta_i ]=h'_2(0)=
\sum_{j=1}^{k_2}p_{2,j}\,\mu
_{2,j}=\mu_2,
\]
\[
\mathbb E \bigl[\xi_i^2 \bigr]=h''_1(0)=
\sum_{j=1}^{k_1}2p_{1,j}\,
\mu_{1,j}^2, \qquad\mathbb E \bigl[\eta_i^2
\bigr]=h''_2(0)=\sum
_{j=1}^{k_2}2p_{2,j}\, \mu_{2,j}^2,
\]
\[
\mathbb E \bigl[\xi_i^3 \bigr]=h'''_1(0)=
\sum_{j=1}^{k_1}6p_{1,j}\,
\mu_{1,j}^3, \qquad\mathbb E \bigl[\eta_i^3
\bigr]=h'''_2(0)=\sum
_{j=1}^{k_2}6p_{2,j}\,\mu_{2,j}^3.
\]

Consequently, we get
\[
\mathbb E \bigl[(\xi_i-\eta_i)^2 \bigr] =2
\Biggl(\sum_{j=1}^{k_1}p_{1,j}\,
\mu_{1,j}^2 -\mu_1\mu_2 +\sum
_{j=1}^{k_2}p_{2,j}\,
\mu_{2,j}^2 \Biggr),
\]\vspace*{-6pt}%
\begin{align*}%
&\mathbb E \bigl[(\xi_i-\eta_i)^3 \bigr]\\
&\quad{} =6
\Biggl(\sum_{j=1}^{k_1}p_{1,j}\,
\mu_{1,j}^3 -\mu_2\sum
_{j=1}^{k_1}p_{1,j}\,\mu_{1,j}^2
+\mu_1\sum_{j=1}^{k_2}p_{2,j}
\,\mu_{2,j}^2 -\sum_{j=1}^{k_2}p_{2,j}
\,\mu _{2,j}^3 \Biggr).
\end{align*}

\begin{ex}
\label{ex:2}
Let $c=10$, $\lambda=4$, $\mu_1=2$, $\mu_2=0{.}5$, $k_1=3$, $k_2=2$,
$p_{1,1}=0{.}4$, $\mu_{1,1}=0{.}5$, $p_{1,2}=0{.}3$, $\mu_{1,2}=2$,
$p_{1,3}=0{.}3$, $\mu_{1,3}=4$,
$p_{2,1}=0{.}75$, $\mu_{2,1}=0{.}4$, $p_{2,2}=0{.}25$, $\mu
_{2,2}=0{.}8$. Then
$\hat R \approx0{.}110607$, which may not be a unique positive
solution to~\eqref{eq:24},
and
\[
\psi_{DV}(x)=0{.}581428\,e^{-0{.}108865\,x}.
\]

The results of computations are given in Table~\ref{table:2}.

\begin{table}
\caption{Results of computations: hyperexponential distributions for
claim sizes and additional funds}\label{table:2}
\begin{tabular*}{\textwidth}{@{\extracolsep{\fill}}D{.}{.}{2.0}D{.}{.}{1.6}D{.}{.}{1.6}D{.}{.}{3.14}D{.}{.}{1.6}D{.}{.}{2.12}@{}}
\hline
\multicolumn{1}{@{}l}{$x$} &
\multicolumn{1}{l}{$\hat\psi(x)$} &
\multicolumn{1}{l}{$\psi_{DV}(x)$} &
\multicolumn{1}{l}{$ \bigl(\frac{\psi_{DV}(x)}{\hat\psi(x)}-1\bigr)\cdot100\%$} &
\multicolumn{1}{l}{$e^{-\hat Rx}$} &
\multicolumn{1}{l@{}}{$ \bigl(\frac{e^{-\hat Rx}}{\hat\psi(x)}-1\bigr)\cdot100\%$} \\
\hline
\rule{0pt}{9pt}0 &0.647560 &0.581428 &-10.21\% &1.000000 &54.43\%\\
1 &0.540924 &0.521454 &-3.60\% &0.895291 &65.51\%\\
2 &0.488597 &0.467667 &-4.28\% &0.801545 &64.05\%\\
5 &0.346390 &0.337363 &-2.61\% &0.575201 &66.06\%\\
10 & 0.202323 &0.195749 &-3.25\% &0.330856 &63.53\%\\
20 & 0.067802 &0.065903 &-2.80\% &0.109466 &61.45\%\\
25 & 0.038194 &0.038239 &0.12\% &0.062965 &64.86\%\\
\hline
\end{tabular*}
\end{table}
\end{ex}

\subsection{Exponential distribution for claim sizes and degenerate
distribution for additional funds}

Let $\xi_i$ be exponentially distributed with mean $\mu_1$ and
$\mathbb E  [\eta_i=\mu_2 ]=1$.
Then condition~\eqref{eq:3} can be rewritten as
\[
\frac{\lambda e^{-\mu_2 \hat R}}{1-\mu_1 \hat R} =\lambda+c\hat R,
\]
where $\hat R\in(0,1/\mu_1)$, which is equivalent to
\begin{equation}
\label{eq:25} \lambda e^{-\mu_2 \hat R} =-c\mu_1 \hat
R^2 +(c-\lambda\mu_1)\hat R +\lambda.
\end{equation}

If $c-\lambda\mu_1+\lambda\mu_2 >0$, then it is easy to check
that~\eqref{eq:25}
has a unique solution $\hat R\in(0,1/\mu_1)$.

Since
\[
\mathbb E [\xi_i ]=\mu_1, \qquad\mathbb E \bigl[
\xi_i^2 \bigr]=2\mu_1^2, \qquad
\mathbb E \bigl[\xi_i^3 \bigr]=6\mu_1^3,
\]
\[
\mathbb E [\eta_i ]=\mu_2, \qquad\mathbb E \bigl[
\eta_i^2 \bigr]=\mu_2^2, \qquad
\mathbb E \bigl[\eta_i^3 \bigr]=\mu_2^3,
\]
we get
\[
\mathbb E \bigl[(\xi_i-\eta_i)^2 \bigr] =2
\mu_1^2-2\mu_1\mu_2+
\mu_2^2,
\]
\[
\mathbb E \bigl[(\xi_i-\eta_i)^3 \bigr] =6
\mu_1^3-6\mu_1^2
\mu_2+3\mu_1\mu_2^2-
\mu_2^3.
\]

\begin{ex}
\label{ex:4}
Let $c=10$, $\lambda=4$, $\mu_1=2$, $\mu_2=0{.}5$. Then
$\hat R \approx0{.}195273$ and
\[
\psi_{DV}(x)=0{.}582498\,e^{-0{.}187764\,x}.
\]

The results of computations are given in Table~\ref{table:3}.

\begin{table}
\caption{Results of computations: exponential distribution for claim
sizes and degenerate distribution additional funds}\label{table:3}
\begin{tabular*}{\textwidth}{@{\extracolsep{\fill}}D{.}{.}{2.0}D{.}{.}{1.6}D{.}{.}{1.6}D{.}{.}{3.14}D{.}{.}{1.6}D{.}{.}{2.13}@{}}
\hline
\multicolumn{1}{@{}l}{$x$} &
\multicolumn{1}{l}{$\hat\psi(x)$} &
\multicolumn{1}{l}{$\psi_{DV}(x)$} &
\multicolumn{1}{l}{$\bigl(\frac{\psi_{DV}(x)}{\hat\psi(x)}-1\bigr)\cdot100\%$} &
\multicolumn{1}{l}{$e^{-\hat Rx}$} &
\multicolumn{1}{l@{}}{$\bigl(\frac{e^{-\hat Rx}}{\hat\psi(x)}-1\bigr)\cdot100\%$} \\
\hline
\rule{0pt}{9pt}
0 &0.637998 &0.582498 &-8.70\% &1.000000 &56.74\%\\
1 &0.549737 &0.482780 &-12.18\% &0.822610 &49.64\%\\
2 &0.465171 &0.400133 &-13.98\% &0.676687 &45.47\%\\
5 &0.277026 &0.227808 &-17.77\% &0.376678 &37.97\%\\
10 & 0.113399 &0.089093 &-21.43\% &0.141886 &25.12\%\\
\hline
\end{tabular*}
\end{table}
\end{ex}

\section{Conclusion}
\label{se:6}

Tables~\ref{table:1}--\ref{table:3} provide results of computations when
the initial surplus is not too large. In this case, the statistical estimates
obtained by the Monte Carlo method can be used instead of the exact ruin
probabilities to compare an accuracy of the exponential
bound and the analogue to the De Vylder approximation. To get appropriate
statistical estimates by the Monte Carlo method for large initial surpluses,
the number of simulations must be exceeding.
The results of computations show that the exponential bound is very rough.
The analogue to the De Vylder approximation gives much more accurate
estimations, especially in the case of hyperexponential distributions
for claim sizes and additional funds. Nevertheless, it is heuristic, and
its real accuracy is unknown.

\end{document}